\documentclass [12pt,a4paper,reqno]{amsart}

\usepackage{amsmath,amsfonts,amssymb,amscd,verbatim,delarray,verbatim,url}

\DeclareMathOperator{\const}{const}

\DeclareMathOperator{\Aut}{Aut}

\chardef\bslash=`\\ 

\begin{document}

\numberwithin{equation}{section}

\newtheorem{Theorem}{Theorem}[section]

\newtheorem{cor}[Theorem]{Corollary}

\newtheorem{Conjecture}[Theorem]{Conjecture}

\newtheorem{Lemma}[Theorem]{Lemma}
\newtheorem{lemma}[Theorem]{Lemma}
\newtheorem{property}[Theorem]{Property}
\newtheorem{Proposition}[Theorem]{Proposition}
\newtheorem{ax}[Theorem]{Axiom}
\newtheorem{claim}[Theorem]{Claim}

\theoremstyle{definition}
\newtheorem{Definition}[Theorem]{Definition}
\newtheorem{problem}[Theorem]{Problem}
\newtheorem{question}[Theorem]{Question}
\newtheorem{Example}[Theorem]{Example}

\newtheorem{remark}[Theorem]{Remark}
\newtheorem{diagram}{Diagram}
\newtheorem{Remark}[Theorem]{Remark}
\newcommand{\diagref}[1]{diagram~\ref{#1}}
\newcommand{\thmref}[1]{Theorem~\ref{#1}}
\newcommand{\secref}[1]{Section~\ref{#1}}
\newcommand{\subsecref}[1]{Subsection~\ref{#1}}
\newcommand{\lemref}[1]{Lemma~\ref{#1}}
\newcommand{\corref}[1]{Corollary~\ref{#1}}
\newcommand{\exampref}[1]{Example~\ref{#1}}
\newcommand{\remarkref}[1]{Remark~\ref{#1}}
\newcommand{\corlref}[1]{Corollary~\ref{#1}}
\newcommand{\claimref}[1]{Claim~\ref{#1}}
\newcommand{\defnref}[1]{Definition~\ref{#1}}
\newcommand{\propref}[1]{Proposition~\ref{#1}}
\newcommand{\prref}[1]{Property~\ref{#1}}
\newcommand{\itemref}[1]{(\ref{#1})}


\newcommand{\CE}{\mathcal{E}}
\newcommand{\CG}{\mathcal{G}}\newcommand{\CV}{\mathcal{V}}
\newcommand{\CL}{\mathcal{L}}
\newcommand{\CM}{\mathcal{M}}
\newcommand{\A}{\mathcal{A}}
\newcommand{\CO}{\mathcal{O}}
\newcommand{\B}{\mathcal{B}}
\newcommand{\CS}{\mathcal{S}}
\newcommand{\CX}{\mathcal{X}}
\newcommand{\CY}{\mathcal{Y}}
\newcommand{\CT}{\mathcal{T}}
\newcommand{\CW}{\mathcal{W}}
\newcommand{\CJ}{\mathcal{J}}

\newcommand{\st}{\sigma}
\renewcommand{\k}{\varkappa}
\newcommand{\Frac}{\mbox{Frac}}
\newcommand{\X}{\mathcal{X}}
\newcommand{\wt}{\widetilde}
\newcommand{\wh}{\widehat}
\newcommand{\mk}{\medskip}
\renewcommand{\sectionmark}[1]{}
\renewcommand{\Im}{\operatorname{Im}}
\renewcommand{\Re}{\operatorname{Re}}
\newcommand{\la}{\langle}
\newcommand{\ra}{\rangle}
\newcommand{\LND}{\mbox{LND}}
\newcommand{\Pic}{\mbox{Pic}}
\newcommand{\lnd}{\mbox{lnd}}
\newcommand{\GLND}{\mbox{GLND}}\newcommand{\glnd}{\mbox{glnd}}
\newcommand{\Der}{\mbox{DER}}\newcommand{\DER}{\mbox{DER}}
\renewcommand{\th}{\theta}
\newcommand{\ve}{\varepsilon}
\newcommand{\1}{^{-1}}
\newcommand{\iy}{\infty}
\newcommand{\iintl}{\iint\limits}
\newcommand{\capl}{\operatornamewithlimits{\bigcap}\limits}
\newcommand{\cupl}{\operatornamewithlimits{\bigcup}\limits}
\newcommand{\suml}{\sum\limits}
\newcommand{\ord}{\operatorname{ord}}
\newcommand{\bk}{\bigskip}
\newcommand{\fc}{\frac}
\newcommand{\g}{\gamma}
\newcommand{\be}{\beta}
\newcommand{\dl}{\delta}
\newcommand{\Dl}{\Delta}
\newcommand{\lm}{\lambda}
\newcommand{\Lm}{\Lambda}
\newcommand{\om}{\omega}
\newcommand{\ov}{\overline}
\newcommand{\vp}{\varphi}
\newcommand{\kap}{\varkappa}

\newcommand{\Vp}{\Phi}
\newcommand{\Varphi}{\Phi}
\newcommand{\BC}{\mathbb{C}}
\newcommand{\C}{\mathbb{C}}\newcommand{\BP}{\mathbb{P}}
\newcommand{\BQ}{\mathbb {Q}}
\newcommand{\BM}{\mathbb{M}}
\newcommand{\BR}{\mathbb{R}}\newcommand{\BN}{\mathbb{N}}
\newcommand{\BZ}{\mathbb{Z}}\newcommand{\BF}{\mathbb{F}}
\newcommand{\BA}{\mathbb {A}}
\renewcommand{\Im}{\operatorname{Im}}
\newcommand{\id}{\operatorname{id}}
\newcommand{\ep}{\epsilon}
\newcommand{\tp}{\tilde\partial}
\newcommand{\doe}{\overset{\text{def}}{=}}
\newcommand{\supp} {\operatorname{supp}}
\newcommand{\loc} {\operatorname{loc}}
\newcommand{\de}{\partial}
\newcommand{\z}{\zeta}
\renewcommand{\a}{\alpha}
\newcommand{\G}{\Gamma}
\newcommand{\der}{\mbox{DER}}

\newcommand{\Spec}{\operatorname{Spec}}

\newcommand{\tG}{\widetilde G}

\newcommand{\Eq}{\Longleftrightarrow}

\newcommand{\RSL}{\operatorname{SL}}
\newcommand{\PGL}{\operatorname{PGL}}
\newcommand{\RM}{\operatorname{M}}
\newcommand{\rk}{\operatorname{rk}}
\newcommand{\codim}{\operatorname{codim}}
\newcommand{\diag}{\operatorname{diag}}
\newcommand{\ad}{\operatorname{ad}}

\newcommand{\FX}{\mathfrak {X}}
\newcommand{\FV}{\mathfrak {V}}
\newcommand{\Fg}{\mathfrak {g}}
\newcommand{\Fh}{\mathfrak {h}}
\newcommand{\Fb}{\mathfrak {b}}
\newcommand{\Fn}{\mathfrak {n}}
\newcommand{\Fa}{\mathfrak {a}}
\newcommand{\Fz}{\mathfrak {z}}

\newcommand{\Fsl}{\mathfrak {sl}}
\newcommand{\Fpsl}{\mathfrak {psl}}
\newcommand{\Fso}{\mathfrak {so}}
\newcommand{\Fgl}{\mathfrak {gl}}

\newcommand{\SX}{\mathcal {X}}
\newcommand{\SV}{\mathcal {V}}
\newcommand{\SO}{\mathcal {O}}
\newcommand{\SD}{\mathcal {D}}
\newcommand{\Sr}{\rho}
\newcommand{\SR}{\mathcal {R}}

\newcommand{\SL}{\mathcal {L}}
\newcommand{\SF}{\mathcal {F}}

\def\toeq{{\stackrel{\sim}{\longrightarrow}}}

\title {Equations in simple Lie algebras}

\author[Bandman, Gordeev, Kunyavski\u\i , Plotkin] {Tatiana Bandman,
Nikolai Gordeev, Boris Kunyavski\u\i , Eugene Plotkin}
\address{Bandman: Department of
Mathematics, Bar-Ilan University, 52900 Ramat Gan, ISRAEL}
\email{bandman@macs.biu.ac.il}

\address{Gordeev: Department of Mathematics, Herzen State
Pedagogical University, 48 Moika Embankment, 191186, St.Petersburg,
RUSSIA} \email{nickgordeev@mail.ru}

\address{Kunyavski\u\i : Department of
Mathematics, Bar-Ilan University, 52900 Ramat Gan, ISRAEL}
\email{kunyav@macs.biu.ac.il}

\address{Plotkin: Department of Mathematics, Bar-Ilan University, 52900 Ramat Gan,
ISRAEL} \email{plotkin@macs.biu.ac.il}

\begin{abstract}
Given an element $P(X_1,\dots ,X_d)$ of the finitely generated free
Lie algebra $\SL_d$, for any Lie algebra $\Fg$ we can consider the
induced polynomial map $P\colon \Fg^d\to\Fg$. Assuming that $K$ is
an
arbitrary field 
of characteristic $\ne 2$, we prove that if $P$ is
not an identity in $\Fsl(2,K)$, then this map is dominant for any
Chevalley algebra $\Fg$. This result can be viewed as a weak
infinitesimal counterpart of Borel's theorem on the dominancy of the
word map on connected semisimple algebraic groups.

We prove that for the Engel monomials $[[[X,Y],Y],\dots ,Y]$ and,
more generally, for their linear combinations, this map is,
moreover, surjective onto the set of noncentral elements of $\Fg$
provided that the ground field $K$ is big enough, and show that for
monomials of large degree the image of this map contains no nonzero
central elements.

We also discuss consequences of these results for polynomial maps of
associative matrix algebras.
\end{abstract}

\maketitle

\section {Introduction}\label{sec:intro}

For a given element $P(X_1,\dots ,X_d)$ of the finitely generated
free Lie algebra $\SL_d=\SL(X_1,\dots ,X_d)$ defined over a given
field $K$, and a given Lie algebra $\Fg$ over $K$, one can ask the
following question:

\begin{question} \label{quest-lie}
Is the equation
$$
P(X_1,\dots ,X_d)=A
$$
solvable

a) for all $A\in \Fg$,

\noindent or, at least,

b) for a generic $A\in \Fg$?
\end{question}

In the present paper we consider this question in a particular case
of {\it ``classical''} semisimple Lie algebras, i.e., quotients of
Chevalley algebras modulo the centre, see Section \ref{2.1} for
precise definitions. (We are abusing the terminology accepted in the
characteristic zero case in order to distinguish from Lie algebras
of Cartan type appearing in positive characteristics.)

Our motivation is two-fold. The primary inspiration came from widely
discussed group-theoretic analogues of Question \ref{quest-lie}:

\begin{question} \label{quest-group}
Let $w(x_1,\dots ,x_d)$ be an element of the finitely generated free
group $\SF_d=\SF (x_1,\dots ,x_d)$ (i.e., a word in $x_i$ and
$x_i^{-1}$), and let a group $G$ be given. Is the equation
$$
w(x_1,\dots ,x_d)=g
$$
solvable

a) for all $g\in G$,

\noindent or, at least,

b) for a generic $g\in G$?
\end{question}

If $G$ is a connected semisimple algebraic $K$-group, a theorem of
Borel \cite{Bo2}, stating that the word map $G^d\to G$ is dominant
whenever $w\ne 1$, gives a positive answer to part b). One can,
however, easily produce examples where the word map is not
surjective and so the answer to part a) is negative (see \cite{Bo2}
and references therein). Some particular words have been extensively
studied, and Question \ref{quest-group}a has been answered in the
affirmative. Say, if $d=2$ and $w(x,y)=[x,y]$ (the commutator), the
positive answer is known long ago for the connected semisimple
compact topological groups \cite{Got}, connected complex semisimple
Lie groups \cite{PW} and algebraic groups defined over an
algebraically closed field \cite{Ree}, as well as for some simple
groups over reals \cite{Dj} and more general fields \cite{Th1},
\cite{Th2}. In the case where $G$ is a finite (nonabelian) simple
group, Question \ref{quest-group} for this word constitutes an old
problem posed by Ore in 1950s and solved very recently in
\cite{LOST}; on the way to this solution, part b) was thoroughly
investigated, and several different approaches to the definition of
a ``generic'' element have been tried, see, e.g., \cite{Gow},
\cite{EG}, \cite{Sh1} (the last paper contains a survey of some
recent developments).

Similar questions for word maps on simple groups for words more
complicated than the commutator remain widely open (see, however,
\cite{GS}, \cite{Sh2} for new approaches towards part b), and
\cite{BGG}, where a particular case $G=PSL(2,q)$ and $w$ an Engel
word is treated).

The second motivating example is associative algebras, for which
analogues of Questions \ref{quest-lie} and \ref{quest-group} have
also been intensely investigated (first questions of such kind go
back to Kaplansky), see \cite{KBMR} and references therein. In a
sense, the case of Lie algebras treated in the present paper may be
viewed as a sort of ``bridge'' between groups and associative
algebras in what concerns dominancy and surjectivity of polynomial
maps; see Remark \ref{rem:ass} below.

Going over to infinitesimal analogues, one can first mention that
the equation $[X,Y]=A$ is solvable for all $A$ in any classical
split semisimple Lie algebra $\Fg$, under the assumption that the
ground field $K$ is sufficiently large. (Here, of course, brackets
stand for the Lie product.) This fact was established by Brown
\cite{Br}, and in \cite{Hi} estimates on the size of $K$ were
improved.

Our first results (Section \ref{sec:dom}) concern the general case
where we are given an element $P(X_1,\dots ,X_d)$ of the finitely
generated free Lie algebra $\SL_d$ defined over a field $K$. Then
for any Lie algebra $\Fg$ defined over $K$ we can consider the
induced polynomial map $P\colon \Fg^d\to\Fg$. Assuming that $K$ is
an arbitrary field of characteristic $\ne 2$, we prove that if $P$
is not an identity in $\Fsl(2,K)$, then this map is dominant for any
Chevalley algebra $\Fg$. This result can be viewed as a weak
infinitesimal counterpart of Borel's theorem on the dominancy of the
word map on connected semisimple algebraic groups.

Going over from dominancy to surjectivity (Section \ref{sec:sur}),
we prove that for the Engel monomials $[[[X,Y],Y],\dots ,Y]$ and,
more generally, for their linear combinations, the image of the
corresponding map contains the set of noncentral elements of $\Fg$
provided that the ground field $K$ is big enough. We show that for
Engel monomials of large degree this image contains no nonzero
central elements.

We also discuss consequences of these results for polynomial maps of
associative matrix algebras as well as some other possible
generalizations (Section \ref{sec:gen}).

\section{Preliminaries}

Our notation is standard. In particular, $\mathbb Q$ 
and ${\mathbb F}_p$ denote the field of rational numbers 
and the finite field of $p$ elements, respectively. $\mathbb A^n$
denotes the $n$-dimensional affine space. The cardinality of a set
$A$ is denoted by $|A|$. We use the Zariski topology throughout, and
$\bar A$ denotes the closure of $A$. The group of automorphisms of
any object $X$ (group, algebra, variety) is denoted by $\Aut (X)$.
The orbit of an element $h$ of a set $H$ with respect to an action
of a group $W$ is denoted by $Wh$, and $WH'$ denotes the union of
$W$-orbits of all elements of $H'\subseteq H$. If $R$ is a ring on
which a group $G$ acts, $R^G$ denotes the ring of invariants. The
reader is referred to \cite{Bou} and \cite{Hu} for all unexplained
notions and facts concerning Lie algebras.

\subsection{Dominant maps} \label{subsec:dom}
Recall that a $K$-morphism $f\colon V\to W$ of algebraic
$K$-varieties (=reduced $K$-schemes of finite type) is called {\it
dominant} if its image $f(V)$ is Zariski dense in $W$. We will
mostly deal with the case where $V$ and $W$ are irreducible. In such
a case $f(V)$ contains a non-empty open set $U$ (see, e.g.,
\cite[Th.~IV.3.7]{Pe}). If $L/K$ is a field extension, then $f$ is
dominant if and only if the $L$-morphism $f_L\colon V_L\to W_L$
obtained by extension of scalars is dominant.

\subsection{Chevalley and classical Lie algebras} \label{2.1}
Let $\sf R$ be a root system and let $\Pi$ be a simple root system
corresponding to $\sf R$. Further, let $L(\sf R, \C)$ be a
semisimple complex Lie algebra. Then there exists a Chevalley basis
$\{h_\alpha\}_{\alpha\in \Pi}\cup \{e_\beta\}_{\beta \in \sf R}$ of
$L(\sf R,\C)$ such that

1) $[e_\alpha, e_{-\alpha}] = h_\alpha$ for every $\alpha \in \Pi$;

2) $h_\beta :=[e_\beta, e_{-\beta}]\in \sum_{\alpha\in \Pi} \mathbb
Zh_\alpha $ for every $\beta \in \sf R$;

3) $[h_\beta, h_{\gamma}] = 0$ for every $\beta, \gamma \in\sf R$;

4) $[h_\beta, e_\gamma] = q_{\beta ,\gamma}e_\gamma$ for every
$\beta, \gamma\in \sf R$ (note that $q_{\beta ,\gamma} = 0, \pm 1,
\pm2, \pm 3$);

5) $[e_\beta, e_\gamma] = 0$ if $\beta +\gamma \notin \sf R$;

6) $[e_\beta, e_\gamma] = p_{\beta,\gamma}e_{\beta+\gamma}$ if
$\beta +\gamma\in \sf R$ (note that $p_{\beta,\gamma} = \pm 1,\pm
2,\pm 3$).

One can now define the corresponding Lie algebra over any prime
field $F$ using the same basis and relations 1)--6) in the case $F =
\mathbb Q$ or the same basis and relations 1)--6) modulo $p$ in the
case $F={\mathbb F}_p$. Then one can define the Lie algebra $L({\sf
R}, K)$ over any field $K$ using the same basis and relations
induced by 1)--6). We will denote such an algebra by $L({\sf R}, K)$
and call it a {\it Chevalley } algebra. The Chevalley algebra
$L({\sf R}, K)$ decomposes into the sum $\sum_{i}L({\sf R}_i,K)$
where ${\sf R} = \cup_{i}{\sf R}_i$ is the decomposition of the root
system $\sf R$ into the disjoint union of irreducible root
subsystems. The Lie algebras $L({\sf R}_i, K)$ are not simple if the
characteristic of $K$ is not a ``very good prime'' \cite{Ca}.
However, if ${\sf R}_i \ne {\sf A}_1, {\sf B}_r, {\sf C}_r, {\sf
F}_4$ when $\operatorname{char}(K) = 2$ and ${\sf R}_i \ne {\sf
G}_2$ when $\operatorname{char}(K) = 3$, the algebra $L({\sf R}_i,
K)/Z_i$ is simple (here $Z_i$ is the centre of $L({\sf R}_i, K)$).
Thus, if the Lie algebra $L({\sf R}_i, K)$ has no components pointed
out above, the quotient $L({\sf R}, K)/Z$ (where $Z$ is the centre
of $L({\sf R},K)$) is a semisimple Lie algebra.

\medskip

Let $L({\sf R}, K)$ be a Chevalley algebra over a field $K$ which
corresponds to an irreducible reduced root system $\sf R$. Denote by
$\sf R^+$ (resp. $\sf R^-$) the set of positive (resp. negative)
roots and put
$$H= \sum_{\alpha \in \Pi} Kh_\alpha, \,\,\, U^{\pm} = \sum_{\beta \in \sf R^{\pm}}Ke_\beta,\,\,\, U = U^-+U^+.$$
Then
$$L({\sf R}, K) = H + U = H + U^- + U^+.$$
The number $r = \dim H = |\Pi|$ is called the {\it rank } of $L({\sf
R}, K)$.

Let now $i\colon L({\sf R}, K) \rightarrow \operatorname{End}(V)$ be
a linear representation. Then one can construct the corresponding
Chevalley group $G({\sf R},K)\leq GL(V)$, which is generated by the
so-called root subgroups $x_\beta(t)$ (see \cite{St2}, \cite{Bo1}),
and a homomorphism $j\colon G({\sf R},K) \rightarrow
\operatorname{Aut}(i(L({\sf R}, K)))$.

Suppose $K$ is an algebraically closed field and $i$ is a
representation of $L({\sf R}, K)$ such that the group of weights of
$i$ coincides with the group generated by fundamental weights. Then
$G({\sf R},K)$ is a simple, simply connected algebraic group,
$i(L({\sf R}, K))$ is the Lie algebra of $G({\sf R},K)$, and the
homomorphism $j$ defines the adjoint action of $G({\sf R},K)$ on its
Lie algebra $i(L({\sf R}, K))$ \cite[3.3]{Bo1}.

Below we will always consider the Chevalley group $G({\sf R},K)$
constructed through a faithful representation $i$ such that $G({\sf
R},K)$ is simply connected. We also identify the Lie algebra
$i(L({\sf R}, K))$ with $L({\sf R}, K)$. The group $j(G({\sf
R},K))\leq \operatorname{Aut}(L({\sf R}, K))$ will be denoted by
$G$.  Note that $G$ is the group generated by the images
$j(x_\beta(t))$ of the root subgroups which will also be denoted by
$x_\beta(t)$.

An element $x \in L({\sf R}, K)$ is called {\it semisimple} (resp.
{\it nilpotent}) if for a faithful linear representation $\rho\colon
L({\sf R}, K)\rightarrow \operatorname{End}(V)$   the operator
$\rho(x)$  is  semisimple (resp. nilpotent). Every $x \in L({\sf R},
K)$ has the Jordan decomposition $x = x_s + x_n$ where $x_s$ is
semisimple, $x_n$ is nilpotent, $[x_s,x_n]=0$.

Let $K$ be an algebraically closed field. Then:

{\bf a.} {\it  Every element of the Lie algebra $L({\sf R}, K)$ is
$G$-conjugate to an element $x = x_s + x_n$ such that $x_s \in H$,
$x_n\in U^+$, $[x_s,x_n]=0$}.

\medskip

{\bf b.} \cite[II.3.20]{SS} {\it  The $G$-orbit of an element $x \in
L({\sf R}, K)$ is closed if and only if $x$ is semisimple.}

\medskip

{\bf c.} {\it For every root $\beta \in {\sf R}$ there is a linear
map $\beta \colon  H\rightarrow K$ defined by the formula $[h,
e_\beta] = \beta(h)e_\beta$. Suppose there is a regular element $h
\in H$, i.e., $\beta(h) \ne 0$ for every $\beta \in \sf R$. Then the
set of all elements in $L({\sf R}, K)$ which are $G$-conjugate to
elements from $H$ is dense in $L({\sf R},K)$}.

\medskip

{\bf d.} {\it  There is a $G$-equivariant dominant morphism
$$\pi \colon L({\sf R}, K)\rightarrow \sf Y$$
where $ \sf Y$ is an affine variety and the map $$\bar{\pi} =
\pi_{\mid H}\colon H\rightarrow \sf Y$$ satisfies the following
condition:
$$\bar{\pi}^{-1}(\pi(h)) = Wh$$
where $W$ is the Weyl group, which acts naturally on $H$}.

\begin{proof}
Put $L = L({\sf R}, K)$, and let $S = K[L]$ be the algebra of
polynomial functions on $L$ (i.e., the symmetric algebra of the
underlying vector space of $L$). Since $G$ is a simple algebraic
group, $R = S^G$ is finitely generated (see, e.g.,
\cite[Cor.~2.4.10]{Sp}), say, by $f_1, \ldots, f_k$. Consider the
map
$$\pi\colon L\rightarrow {\mathbb A}^k$$
given by the formula $\pi(x) = (f_1(x), \ldots, f_k(x))$. If $x =
x_s +x_n$ is the Jordan decomposition then $\pi(x) = \pi (x_s)$.
Indeed, the $G$-orbit $O_{x_s}$ of $x_s$ is contained in the Zariski
closure $\overline{O_x}$ of the $G$-orbit of $x$ (the fact that the
semisimple part of an element $y$ of a semisimple group $G$ is
contained in the closure of the conjugacy class is well known, see
\cite[II, 3.7]{SS}; for $G$-orbits of the Lie algebra of $G$ the
proof of the corresponding part is almost the same, see, e.g.,
\cite[Prop.~2.11]{Ja}).

Since $\pi$ is a regular map constant on the orbit $O_x$, it is
constant on $\overline{O_x}$. Hence ${\sf
Y}:=\overline{\operatorname{Im}\,\pi}
=\overline{\operatorname{Im}\,\bar{\pi}}$. Further, functions from
$R$ separate closed orbits in $L$ (see, e.g., \cite[Chapter~1,
$\S~1.2$]{Po}). Hence $\bar{\pi}^{-1}(\pi (h)) = H \cap O_h$ where
$O_h$ is the orbit of $h$. Since $H \cap O_h = Wh$ \cite[3.16]{SS},
we are done.
\end{proof}

\begin{remark} If $\operatorname{char}(K)$ is not a torsion prime for
$G({\sf R}, K)$, then there is an isomorphism $\pi^\prime\colon
L({\sf R}, K)/G\toeq H/W$, and the quotient $H/W$ is isomorphic to
${\mathbb A}^r$ \cite{SS}, \cite[3.12]{Sl}. Hence in this case ${\sf
Q}\cong H/W\cong {\mathbb A}^r$. In general, the natural morphism
$H/W \to L({\sf R}, K)/G$, induced by the inclusion $H\to L({\sf R},
K)$, is dominant; it is an isomorphism if and only if ${\sf R} =
{\sf C}_r$, $r\ge 1$, $\operatorname{char}(K) = 2$ \cite{CR}.
\end{remark}

\subsection{Cartan subalgebras and regular elements} \label{2.2}

Let $\beta\in{\sf R}$. We have $($cf. \cite[Lemma 2.3.2]{CR}$)$

\begin{equation}
\beta \equiv 0 \Leftrightarrow  {\sf R = C}_r, r \geq 1,
\operatorname{char}(K) = 2, \beta\,\,\,\text{is a long
root}\label{eq2.1}
\end{equation}
$($here ${\sf C}_1 = {\sf A}_1$, ${\sf C}_2 = {\sf B}_2)$. Thus, if
we are not in the case ${\sf R = C}_r$, $r \geq 1$,
$\operatorname{char}(K) = 2$, the subalgebra $H$ is a {\it  Cartan}
subalgebra, that is, a nilpotent subalgebra coinciding with its
normalizer. In the case ${\sf R = C}_r$, $r \geq 1$,
$\operatorname{char}(K) = 2$, the subalgebra $H$ is a Cartan
subalgebra of $[L({\sf R}, K), L({\sf R}, K)]\cong L({\sf D}_r, K)$
$($here $L({\sf D}_1, K) = K, L({\sf D}_2, K) = \mathfrak{sl}(2,
K)\times \mathfrak{sl}(2, K))$.

\bigskip

{\bf e.} {\it Suppose we are not in the case ${\sf R = C}_r$, $r
\geq 1$, $\operatorname{char}(K) = 2$. Then if $|K| \geq |\sf R^+|$,
the subalgebra $H$ contains a regular element. Moreover, if $|K| > m
|\sf R|$ for some $m \in \mathbb N$, then for every subset $S\subset
K$ of size $m$ there exists $h \in H$ such that $\beta(h)\notin S$
for every $\beta \in \sf R$.}

\begin{proof}
For infinite fields the statement is trivial. If $K$ is a finite
field, then $|H| = |K|^r$, and the hyperplane $H_{x, \beta}$ of $H$
defined by the equation $[h, e_\beta] = x$, $x\in S$, consists of
$|K|^{r-1}$ points. Thus,
$$
|\bigcup_{x \in S, \beta \in {\sf R}}H_{x,\beta}| \leq |S| \cdot
|H_{x,\beta}| = m|{\sf R}| \cdot |K|^{r-1} < |H|,
$$
and therefore we can take $h \in H\setminus \cup_{x \in S, \beta \in
{\sf R}}H_{x,\beta}$. The first statement can be proved by the same
arguments for $S = \{0\}$ using the fact that  $0 \in H_{0, \beta} =
H_{0, -\beta}$ for every $\beta \in {\sf R}$.
\end{proof}

\subsection{Exceptional cases} \label{2.3}

The Chevalley algebra $L({\sf R}, K)$ modulo the centre is not
simple in the following cases (see, e.g., \cite{Ho}):
\begin{equation}
{\sf R}= {\sf A}_1, {\sf B}_r, {\sf C}_r, {\sf F}_4 {\text{\rm{ if
char}}}(K)=2, \,\,\,{\sf R} = {\sf G}_2 {\text{\rm{ if char}}}(K)=3.
\label{eq2.2}
\end{equation}
Namely:

1) Let  ${\sf R}= {\sf A}_1$  and $\operatorname{char}(K) = 2$. Then
$L({\sf A}_1, K) \cong \mathfrak{sl}(2, K) $ is a nilpotent algebra
satisfying the identity $[[X, Y], Z]\equiv 0$.

2)  Let  ${\sf R}= {\sf B}_2$  and $\operatorname{char}(K) = 2$.
Then $L({\sf B}_2, K) \cong \mathfrak{so}(5, K) $ is  a solvable
algebra satisfying the identity $[[X, Y], [Z, T]]\equiv 0$.

3) Let ${\sf R} = {\sf B}_r, r > 2$ and $\operatorname{char}(K) =
2$. Then $L({\sf B}_r, K)$ contains the nilpotent ideal $I$
generated by $\{e_\beta\mid \,\,\,\beta \,\,\,\text{is a short
root}\}$, and $L({\sf B}_r, K)/I \cong L({\sf D}_r, K)/Z^\prime $
where $Z^\prime \leq Z(L({\sf D}_r, K))$.

4) Let  ${\sf R} = {\sf F}_4$ and $\operatorname{char}(K) = 2$. Then
$L({\sf F}_4, K)$ contains the ideal $I$ generated by $\{e_\beta\mid
\,\,\,\beta \,\,\,\text{is a short root}\}$, and $L({\sf F}_4, K)/I
\cong L({\sf D}_4, K)/Z^\prime $ with $Z^\prime \leq Z(L({\sf D}_4,
K))$ where $Z(L({\sf D}_4, K))$ is the centre.

5)  Let  ${\sf R} = {\sf G}_2$ and $\operatorname{char}(K) = 3$.
Then $L({\sf G}_2, K)$ contains the ideal $I\cong \mathfrak{sl}(3,
K) $ generated by $\{e_\beta\mid \,\,\,\beta \,\,\,\text{is a short
root}\}$, and the algebra $L({\sf G}_2, K)/I$ is isomorphic to
$\mathfrak{sl}(3, K) /Z( \mathfrak{sl}(3, K) ) $.

6) Let ${\sf R} = {\sf C}_r, r > 2$ and $\operatorname{char}(K) =
2$. Then $L({\sf C}_r, K)$ contains the ideal $I\cong L({\sf D}_r,
K)$ generated by $\{e_\beta\mid \,\,\,\beta \,\,\,\text{is a short
root}\}$, and the algebra $L({\sf C}_r, K)/I $ is abelian.

\bigskip

The other Chevalley algebras $L({\sf R}, K)$ corresponding to
irreducible root systems $\sf R$ are simple modulo the centre $Z$
\cite{St1}. The simple algebras $\Fg =L({\sf R}, K)/Z$ are
classical. The classical semisimple Lie algebras and the
corresponding Chevalley algebras form a natural class to consider
polynomial maps $P(X_1, \dots, X_d)$ on its products. However, note
that the algebras appearing in ``bad cases'' 3)--5) are perfect,
i.e., satisfy the condition $[L({\sf R}, K), L({\sf R}, K) ] =
L({\sf R}, K)$, and therefore we can also raise the question on
dominancy of polynomial maps on such algebras.

\subsection{Prescribed Gauss decomposition} \label{2.5}

We will use the following generalization of the classical Gauss
decomposition.

{\bf f.} {\it Suppose we are not in the cases appearing in list
\eqref{eq2.2}. Fix an arbitrary non-central element $h \in H$. Then
for every non-central element $l \in L({\sf R}, K)$ there is $g \in
G$ such that $g(l) \in h + U$} \cite[Proposition~1]{Gorde} (actually
we mostly need below a particular case $h = 0$ treated in
\cite[Lemma II]{Br}).

\section{Dominancy of polynomial maps on Chevalley algebras}
\label{sec:dom}

In this section $K$ is an algebraically closed field.

\subsection{}

We are interested in the following analogue of the Borel dominancy
theorem for semisimple Lie algebras:

\begin{question} \label{quest:dom}
For a given element $P(X_1,\dots ,X_d)$ of the free Lie $K$-algebra
$\SL_d$ on the finite set $\{X_1,\dots ,X_d\}$ over a given
algebraically closed field $K$, and a given semisimple Lie algebra
$\Fg$ over $K$, is the map
$$P(X_1, \ldots, X_d)\colon  \Fg^d\rightarrow \Fg$$
dominant under the condition that $P(X_1, \ldots, X_d)$ is not an
identity on $\Fg$?
\end{question}

We do not know the answer to this question. However, we can get it
under some additional assumption. Our main result is

\begin{Theorem} \label{th:main}
Let $L({\sf R},K)$ be a Chevalley algebra. If
$\operatorname{char}(K) = 2$, assume that $\sf R$ does not contain
irreducible components of type ${\sf C}_r$, $r \geq 1$ $($here ${\sf
C}_1 = {\sf A}_1, {\sf C}_2 = {\sf B}_2)$.

Suppose $P(X_1,\dots ,X_d)$ is not an identity of the Lie algebra
$\mathfrak{sl}(2, K)$. Then the induced map $P\colon L({\sf
R},K)^d\to L({\sf R},K)$ is dominant.
\end{Theorem}

Below we repeatedly use the following construction. Put
$$\mathfrak{I} = \overline{P(L({\sf R}, K)^d}).$$ Then
$\mathfrak{I}$ is an irreducible affine variety which is
$G$-invariant, and therefore $\overline{\pi
(\mathfrak{I})}\hookrightarrow \sf Y$ is an irreducible closed
subset of an $r$-dimensional affine variety $\sf Y$ (see
\ref{2.1}.{\bf d}). If $\overline{\pi (\mathfrak{I})} = \sf Y$,
then, by \ref{2.1}.{\bf d}, the set $\mathfrak{I}$ contains all
elements which are $G$-conjugate to elements of $H$ (because $\pi(x)
= \pi(x_s)$, $O_{x_s}\subset \overline{O_x}$ and $\mathfrak{I}$ is
$G$-invariant). This implies, by \ref{2.1}.{\bf c}, that
$\mathfrak{I} = L({\sf R}, K)$. Thus,
\begin{equation}
P\,\,\,\,\text{is dominant}\Leftrightarrow \overline{\pi
(\mathfrak{I})} = \sf Y. \label{eq3.1}
\end{equation}

\begin{lemma} \label{lem3.1}
Let $M \subset L({\sf R}, K)$ be an irreducible closed subset such
that
\begin{itemize}
\item[{\rm{(i)}}] $\dim {\overline{\pi (P(M^d))}} = r-1$;
\item[{\rm{(ii)}}] $\overline{\pi(P(M^d))} \ne \overline{\pi (\mathfrak{I})}$.
\end{itemize}
Then $P$ is dominant.
\end{lemma}

\begin{proof}
Since $M$ is irreducible and $\dim {\overline{\pi (P(M^d))}} = r-1$,
we conclude that $\overline{\pi(P(M^d))}$ is an irreducible
hypersurface in $\sf Y$. The assertion of the lemma now follows from
(ii) and \eqref{eq3.1}.
\end{proof}

We can now start the proof of \thmref{th:main}.

First of all, the statement is obviously reduced to the case where
$\sf R$ is irreducible. Indeed, if $\sf R$ is a disjoint union of
${\sf R}_i$, then $\Fg =L({\sf R}, K)$ is a direct sum of
$\Fg_i=L({\sf R}_i, K)$, and the image $\operatorname{Im}\,P$ of the
map $P= P(X_1, \dots, X_d)\colon \Fg ^d \rightarrow \Fg$ is equal to
$\oplus_{i}{\operatorname{Im} P_i}$ where $P_i$ is the restriction
of $P$ to $\Fg_i$.

Note that
$$P(X_1, \dots, X_d) = \sum_i a_i X_i + \sum (\text{monomials from }\,\,\,\SL_d\,\,\,\text{of degree} > 1)$$
where $a_i \in K$. If $a_i \ne 0$ for some $i$, the statement is
trivial. Thus we may and will assume $a_i = 0$ for every $i$.

First we prove the assertion of the theorem for the case ${\sf R =
A}_r$ by induction on $r$.

By \eqref{eq3.1}, it is enough to prove
$\overline{\pi(\mathfrak{I})} = \sf Y$, i.e., $\dim
\overline{\pi(\mathfrak I)}=r$. Note that in the case
$\operatorname{char}(K) = 2$ the statement of the theorem fails for
$r =1$. However, as we will see below, the equality
$\overline{\pi(\mathfrak{I})} = \sf Y$ holds even in this case. Thus
we can prove that $\overline{\pi(\mathfrak{I})} = \sf Y$ by
induction on the rank $r$ starting at $r = 1$.

\bigskip

We identify $L({\sf A_r}, K)$ with $\mathfrak{sl}(r+1,K)$, the
algebra of $(r+1)\times (r+1)$-matrices with zero trace. We fix the
chain of subalgebras $L_1\subset L_2 \subset \cdots \subset L_r =
\mathfrak{sl}(r+1,K)$ where $L_{i-1} = \mathfrak{sl}(i, K)$ is the
subalgebra embedded in the $i\times i$ upper left corner of the
matrix algebra $L_i = \mathfrak{sl}(i+1, K)$. We also fix, for each
$i$, the subalgebra $H_i\subset L_i$ of diagonal matrices in $L_i$.
Further, let $P_i = P_{\mid L_i^d}$.

\bigskip

Induction base: we prove that $\dim \overline{\pi
(\operatorname{Im}\, P_1)} = 1$. Note that according to our
assumption $P_1\ne 0$.

Let first $\operatorname{char}(K) = 2.$ Then according to our
assumption and case 1) of Section \ref{2.3}, we have
$$P_1 =  \sum_{i,j} a_{ij}[X_i, X_j]$$
where $a_{ij} \in K$. Let $a_{i_0,j_0}\ne 0$. Set
$P^\prime_1(X_1,\dots ,X_d)=a_{i_0,j_0} [X_{i_0},X_{j_0}]$. On
putting $X_i = 0$ for all $i \ne i_0, j_0$, we see that
$\operatorname{Im}\,P_1 \supseteq
\operatorname{Im}\,P^\prime_1=H_1$, and therefore $\dim
\overline{\pi (\operatorname{Im}\, P_1)} = 1$ (see \ref{2.1}.{\bf
d}). Note that this case requires special consideration only because
we cannot refer to \lemref{lem3.1}, as below.

Let us now assume  $\operatorname{char}(K) \ne 2.$

As the map $P$ is not identically zero, we may apply \lemref{lem3.1}
with $M = 0$. We obtain the dominancy of $P_1\colon L_1^d\rightarrow
L_1$ which implies $\dim \overline{\pi(\operatorname{Im}\, P_1)} =
1$.

\medskip

Inductive step: assume
\begin{equation}
\dim\overline{\pi(\operatorname{Im}\, P_{r-1})} = r-1\label{eq3.2}
\end{equation}
and prove
\begin{equation}
\dim\overline{\pi(\operatorname{Im}\, P_{r})} = r.\label{eq3.2a}
\end{equation}
We have
$$
H_{r-1} = \{x = \operatorname{diag}(\alpha_1, \dots, \alpha_r, 0)\in
M_{r+1}(K)\mid\,\,
 \operatorname{tr}  x = 0\}.
$$
(Here $M_{r+1}(K)$ is the algebra of $(r+1)\times (r+1)$-matrices
over $K$.) According to \ref{2.1}.{\bf d}, we have
\begin{equation}
\overline{\pi (\operatorname{Im}\, P_{r-1})} =\overline{
\bar{\pi}(H_{r-1})}. \label{eq3.3}
\end{equation}
Suppose that
\begin{equation}
\mathfrak{I}\cap H_r \ne WH_{r-1}.\label{eq3.4}
\end{equation}
Then
\begin{equation}
\overline{\pi( \mathfrak{I})} \supseteq \bar{\pi} (\mathfrak{I}\cap
H_r) \ne \overline{\bar{\pi}(H_{r-1})}. \label{eq3.5}
\end{equation}
Condition \eqref{eq3.2} is condition 1) from Lemma \ref{lem3.1} with
$M = L_{r-1}$. Conditions \eqref{eq3.3} and \eqref{eq3.5} give us
condition 2) from the same lemma. Note that the dominancy of $P$
implies $\dim \overline{\pi (\mathfrak{I}}) = r$. Hence we have to
prove that condition \eqref{eq3.4} holds.

We may assume that the transcendence degree of $K$ is sufficiently
large because this does not have any influence on dominancy of $P$.
Then we may also assume that there exist a subfield $F\subset K$ and
a division algebra $D_{r+1}\subset M_{r+1}(K)$ with centre $F$ such
that $D_{r+1}\otimes_{F} K = M_{r+1}(K)$ \cite{DS}, \cite{Bo2}. The
algebra $D_{r+1}$ is dense in $M_{r+1}(K)$. Hence the set $[D_{r+1},
D_{r+1}]$ is dense in $[M_{r+1}(K), M_{r+1}(K)] = \mathfrak{sl}(r+1,
K)$. On the other hand, $[D_{r+1}, D_{r+1}]\subset D_{r+1}$. Thus
the set $S_{r+1} = D_{r+1}\cap \mathfrak{sl}(r+1, K)$ is dense in $
\mathfrak{sl}(r+1, K)$, and therefore the restriction of $P$ to
$S_{r+1}^d$ is not the zero map. Then there exist $s_1, \dots, s_d
\in S_{r+1}$ such that $s = P(s_1, \dots, s_d)\ne 0$. Since  $s_1,
\dots, s_d\in D_{r+1}$, we have $s \in D_{r+1}$. As there are no
nonzero nilpotent elements in division algebras, all elements of
$D_{r+1}$ are semisimple, so we may assume $s \in H_r$. Since $s$
has no zero eigenvalues, $s \notin WH_{r-1}$, and we get
\eqref{eq3.4}. Thus \eqref{eq3.2a} is proven, and the assertion of
the theorem for ${\sf R}={\sf A}_r$ is established.

\bigskip

The general case is a consequence of the following observation
\cite{Bo2}: every irreducible root system $\sf R$ has a subsystem
${\sf R}^\prime$ which has the same rank as $\sf R$ and decomposes
into a disjoint union of irreducible subsystems ${\sf R}^\prime =
\bigcup_{i}{\sf R}^\prime_i$ where each ${\sf R}_i^\prime$ is a
system of type ${\sf A}_{r_i}$.  Hence
$$L^\prime =\bigoplus_{i} L({\sf A}_{r_i}, K) \subset L({\sf R}, K),\,\,\,\sum_i r_i = r.$$
Thus $$\dim \overline{\pi (P(L^\prime))} = r\Rightarrow\overline{\pi
(\mathfrak{I})} = \sf Y,$$ and we get the statement from
\eqref{eq3.1}.

\thmref{th:main} is proved. \qed

\begin{cor} \label{cor:main}
Let $\Fg$ be a classical semisimple Lie algebra. Suppose that a
polynomial $P(X_1,\dots ,X_d)$ is not an identity of the Lie algebra
$\mathfrak{sl}(2, K)$. Then the induced map $P\colon \Fg^d\to \Fg$
is dominant.
\end{cor}

\begin{proof}
Let $\sf R$ be the root system corresponding to $\Fg$. If the
Chevalley algebra $L({\sf R}, K)$ is semisimple, we have $\Fg=L({\sf
R}, K)$, and there is nothing to prove. If $\Fg =L({\sf R}, K)/Z$,
where $Z$ is the centre, the assertion is an immediate consequence
of the following obvious observation: if the Lie polynomial
$P(X_1,\dots ,X_d)$ does not contain terms of degree 1, then the map
$P\colon L({\sf R}, K)^d\to L({\sf R}, K)$ is trivial on $Z$.
\end{proof}

\subsection{}
\thmref{th:main} reduces the problem of dominancy to the class of
maps $P$ which are identically zero on $\mathfrak{sl}(2, K)$. The
following theorem gives another possibility to reduce the problem of
dominancy.

\begin{Theorem} \label{th:induction}
Let $L({\sf R},K)$ be a Chevalley algebra corresponding to an
irreducible root system $\sf R$, and suppose that ${\sf R}\ne {\sf
C}_r$ if $\operatorname{char}(K)  = 2$.

Suppose that the map $P\colon L({\sf R},K)^d\to L({\sf R},K)$ is
dominant for ${\sf R = A}_2$ and ${\sf B}_2$. Then $P$ is dominant
for every $L({\sf R}, K)$, $r > 1.$
\end{Theorem}

\begin{proof}
We prove the theorem by induction on $r$. Let first $r=2$. The cases
${\sf R = A}_2, {\sf B}_2$ are included in the hypothesis, and the
case ${\sf R = G}_2$ is established by the same argument as at the
end of the proof of \thmref{th:main} because ${\sf G}_2$ contains
${\sf A}_2$.

Let now $r>2$, and  make the induction hypothesis:

\noindent {\it  the map $P$ is dominant for every $L({\sf R}, K)$
where $1 <\operatorname{rank}\,{\sf R} < r$}.

We proceed case by case.

\medskip

${\sf R = A}_r$. The induction step is the same as in the proof of
\thmref{th:main}.

\medskip

${\sf R = C}_r$ $(r\ge 3)$ or ${\sf D}_r$ $(r\ge 4)$. 
Let $\Pi = \{\alpha_1 , \dots, \alpha_r \}$ be the simple root
system numerated as in Bourbaki \cite{Bou}. Let $\Pi_1 =\{ \alpha_1,
\dots, \alpha_{r-1}\}$, $\Pi_2 = \{\alpha_2, \dots, \alpha_{r}\}$.
Then ${\sf R}_1 = \left< \Pi_1\right> = {\sf A}_{r-1}$, ${\sf R}_2 =
\left< \Pi_2\right> = {\sf C}_{r-1}$ or ${\sf D}_{r-1}$,
respectively. Let $H_i = H\cap L({\sf R}_i, K)$. There exists $h \in
H_1$ such that $h\notin WH_2$.

Indeed, let $\epsilon_i\colon H\rightarrow K$ be the weights given
by the formula $\epsilon_i (h_{\alpha_k}) = \frac{2 (\epsilon_i,
\alpha_k)}{(\alpha_k, \alpha_k)}$. Then $\epsilon_1 (H_2) = 0$, and
therefore for every $h^\prime \in WH_2$ we have $\epsilon_i
(h^\prime) = 0$ for some $i$. 
On the other hand, since ${\sf R}_1 = {\sf A}_{r-1}$, we can find $h
\in H_1$ such that $\epsilon_i(h) \ne 0$ for every $i$, and
therefore $h\notin WH_2$.

Note that $h \in \overline{P(L({\sf R_1}, K)^d)}$ because $P$ is
dominant on $L({\sf R_1}, K)$ (see the proof of \thmref{th:main}).
Then $h \in \mathfrak{I}$. On the other hand, $\pi(h) \notin
\overline{\pi (P(L({\sf R}_2, K)^d))}$ because $h \notin WH_2$.
Hence we can apply \lemref{lem3.1} with $M = L({\sf R}_2, K)$.

\medskip

${\sf R = B}_r, {\sf F}_4$. Here we have ${\sf D}_r \subset {\sf
R}$. (The respective embeddings are as follows: $\Fso (2r)\subset
\Fso (2r+1)$ is a natural inclusion, and ${\sf D}_4$ embeds into
${\sf F_4}$ as the subsystem consisting of the long roots.)

Then $H\subset \overline{P(L({\sf D}_r, K))}$, and therefore $P$ is
dominant on $L({\sf R}, K)$.

\medskip

${\sf R = E}_r$. Consider the extended Dynkin diagram. We obtain a
needed subsystem of type $\sf A$ by removing one of its vertices. In
each case the diagram is a trident. We remove the $3$-valent vertex
in the case $r=6$, and the tooth of length 1 (the lower vertex
$\alpha_2$ in the Bourbaki notation) in the cases $r=7, 8$. We
obtain subsystems of types ${\sf A}_2\times {\sf A}_2 \times {\sf
A}_2$, ${\sf A}_7$, and ${\sf A}_8$, respectively. Then we use the
same argument as above.
\end{proof}

\subsection{} \label{example}

To use the theorems proven above for practical purposes, the
following simple remarks may be useful.

If $P(X_1, \ldots, X_d) \in \SL_d$ is a polynomial containing a
monomial of degree $< 5$, then the map $P\colon L({\sf R},
K)^d\rightarrow L({\sf R}, K)$ $(\operatorname{char}(K)\ne 2)$ is
dominant.

The reason is that such a polynomial cannot be an identity in $\Fsl
(2,K)$. Indeed, if it were an identity, so would be its homogeneous
component of the lowest degree (because any homogeneous component of
any polynomial identity of any algebra of any signature over any
infinite field is an identity, see \cite[6.4.14]{Ro}). On the other
hand, any identity of the Lie algebra $\Fsl (2,K)$
$(\operatorname{char}(K)\ne 2)$, is an identity of $\Fgl (2,K)$
(because every matrix is a sum of a trace zero matrix and a scalar
matrix, and such an identity lifts to an identity of the associative
matrix algebra $M_2(K)$). The latter one does not contain identities
of degree less than $4$ (which is the smallest degree of the
so-called standard identity satisfied in $M_2$), hence the same is
true for $\Fgl (2)$ (see, e.g., \cite[Remark 6.1.18]{Ro} or
\cite[Exercise 2.8.1]{Ba}). Moreover, a little subtler argument
allows one to show that $\Fsl (2,K)$ does not contain identities of
degree 4 (see, e.g., \cite[Section 5.6.2]{Ba}).

Note that Razmyslov \cite{Ra} found a finite basis for identities in
this algebra (assuming $K$ to be of characteristic zero). Moreover,
it turned out that all such identities are consequence of the single
identity \cite{Fi}:
$$
P = [[[Y,Z],[T,X]],X] + [[[Y,X],[Z,X]],T],
$$
and this result remains true for any infinite field $K$,
$\operatorname{char}(K)\ne 2$ \cite{Va}.

Below we illustrate how one can apply \thmref{th:induction} using
one of the identities appearing in Razmyslov's basis (the reader
willing to deduce this identity from Filippov's one mentioned above
is referred to Section 2 of \cite{Fi}).

\begin{Example} \label{ex:Raz}
The polynomial $[[[[Z,Y],Y],X],Y] - [[[[Z,Y],X],Y],Y]$ appears in
\cite{Ra} as one of the elements of a finite basis of identities in
$\Fsl (2,K)$ ($\operatorname{char}(K) = 0$). Clearly, the polynomial
$$P(X,Y,Z)=[[[[[Z,Y],Y],X],Y],[[[[Z,Y],X],Y],Y]]$$
is also identically zero in $\Fsl (2,K)$. We check dominancy of the
map
$$P\colon L({\sf R}, K)^3\rightarrow L({\sf R}, K)$$
using computations by MAGMA. In view of \thmref{th:induction}, we
have to check dominancy only for ${\sf R = A}_2, {\sf B}_2$.

\bigskip

Consider the map $\pi \colon L({\sf R}, K)\rightarrow \sf Y$ defined
in Section \ref{2.1}.{\bf{d}}. Since $\operatorname{char}(K) = 0$,
we have ${\sf Y}\cong H/W\cong {\mathbb A}^r$, and $\pi = (f_1, f_2,
\dots, f_r)$ where $f_1, f_2, \dots, f_r$ are $G$-invariant
homogeneous polynomials on $L({\sf R}, K)$ which generate the
invariant algebra $K[L({\sf R}, K)]^G\cong K[H]^W$. Moreover, $\deg
f_1\deg f_2\cdots \deg f_r = |W|$ (see Section \ref{2.1}.{\bf{d}}).
In our cases, $r = 2$ and we have $\deg f_1 = 2$, $\deg f_2 = 3$ for
${\sf R = A}_2$ and $\deg f_1 =2$, $\deg f_2 = 4$ for ${\sf R =
B}_2$.

Let $$0\ne D_1 = P(A, B,C), D_2 = P(A^\prime, B^\prime, C^\prime)
\in \mathfrak{I} = \overline{\operatorname{Im} P(L({\sf R},
K))^3}.$$ Since $P$ is a homogeneous map with respect to $X, Y,Z$,
the lines $l_j := KD_j$, $j=1, 2$, also lie in $\mathfrak{I}$, and
the curves $\pi (l_j)$ in the affine space ${\mathbb A}^2$ with
coordinates $(x_1, x_2)$ are defined by equations of the form
\begin{equation}
x_1^{m_1}/x_2^{m_2} = c_j ,\,\,\,\,\text{where   } 
m_1=\deg f_2, m_2=\deg f_1 , c_j=\const. \label{eq3.7}
\end{equation}
Put
\begin{equation}
\theta: = f_1^{m_1}/f_2^{m_2}. \label{eq3.8}
\end{equation}
From \eqref{eq3.7} and \eqref{eq3.8} we get
\begin{equation}
\theta(D_1) \ne \theta(D_2)\Rightarrow \pi (l_1) \ne \pi (l_2).
\label{eq3.9}
\end{equation}
By \lemref{lem3.1} with $M = l_1$, from \eqref{eq3.9} we see that
the inequality
\begin{equation}
\theta(P(A, B,C)) \ne \theta(P(A^\prime, B^\prime, C^\prime))
\label{eq3.10}
\end{equation}
implies the dominancy of $P$.

\medskip

{\it Case ${\sf R = A}_2$}. We may identify $L({\sf A_2}, K) =
\mathfrak{sl}(3, K)$. The characteristic polynomial is $\chi
(t)=t^3+pt+q$ where  $p$ and $q$ can be viewed as
$SL_3(K)$-invariant homogeneous polynomials of degrees $2$ and $3$,
respectively. Therefore $p = f_1$, $q = f_2$. We point out triples
$(A, B, C)$, $(A^\prime, B^\prime, C^\prime)$ satisfying inequality
\eqref{eq3.10} which were found by MAGMA:

$$
A=\left(\begin{matrix} 3 & 1 & 0 \\  1 &  -1 & 1 \\ 0 & 1 & -2
\end{matrix}\right), \quad
B=\left(\begin{matrix} 2 & 1 & 5 \\  0 &  4 & -3 \\ 1 & 0 & -6
\end{matrix}\right), \quad
C=\left(\begin{matrix} 0 & 0 & 2 \\  1 &  0 & 3 \\ 0 & 1 & 0
\end{matrix}\right),
$$
$$
A'=\left(\begin{matrix} 8 & 1 & 0 \\  1 &  -1 & 1 \\ 0 & 1 & -7
\end{matrix}\right), \quad
B'=\left(\begin{matrix} 2 & 3 & 5 \\  0 &  4 & -3 \\ 1 & 0 & -6
\end{matrix}\right), \quad
C'=\left(\begin{matrix} 0 & 0 & 2 \\  1 &  0 & 3 \\ 0 & 1 & 0
\end{matrix}\right).
$$

\medskip

{\it Case ${\sf R = B}_2$}. We may identify $L({\sf R}, K) =
\mathfrak{so}(5, K)$. Consider the embedding $\mathfrak{so}(5,
K)\hookrightarrow \mathfrak{sl}(5, K)$ given by identification of
$\mathfrak{so}(5, K)$ with matrices of the form
$$
\left(\begin{matrix}
0 & b & c \\
-c^t & m & n \\
-b^t & p & -m^t
\end{matrix}\right)
$$
where $m$, $n$, $p$ are $2\times 2$-matrices, and $n$, $p$ are
skew-symmetric (see, e.g., \cite[1.2]{Hu}). The characteristic
polynomial is $\chi(t)=t^5+pt^3+qt$ where $p$ and $q$ can be viewed
as $SO_5(K)$-invariant homogeneous polynomials on $\mathfrak{so}(5,
K)$ of degrees $2$ and $4$, respectively. Hence $f_1 = p$, $f_2 =
q$. We point out triples $(A, B, C)$, $(A^\prime, B^\prime,
C^\prime)$ satisfying inequality \eqref{eq3.10} which were found by
MAGMA:

$$
A=\left(\begin{matrix} 0 & 1 & 2 & 3 & 4\\  -3 &  5 & 6 & 0 & 9\\ -4
& 7 & 8 & -9 & 0 \\ -1 & 0 & 10 & -5 & -7 \\ -2 & -10 & 0 & -6 & -8
\end{matrix}\right), \quad
A'=\left(\begin{matrix} 0 & 5 & -6 & -7 & 8\\  7 &  -1 & 2 & 0 & 2\\
-8 & 3 & 4 & -2 & 0 \\ -5 & 0 & -3 & 1 & -3 \\ 6 & 3 & 0 & -2 & -4
\end{matrix}\right), \quad
$$
$$
B=\left(\begin{matrix} 0 & 4 & 1 & 2 & 3\\  -2 &  -8 & 6 & 0 & -9\\
-3 & -9 & 7 & 9 & 0 \\ -4 & 0 & 10 & 8 & 9 \\ -1 & -10 & 0 & -6 & -7
\end{matrix}\right), \quad
B'=\left(\begin{matrix} 0 & 6 & -7 & 10 & -3\\  -10 &  -8 & -6 & 0 &
5\\ 3 & 1 & 2 & -5 & 0 \\ -6 & 0 & -4 & 8 & -1 \\ 7 & 4 & 0 & 6 & -2
\end{matrix}\right), \quad
$$
$$
C=\left(\begin{matrix} 0 & -1 & 2 & -3 & 4\\  3 &  -5 & -6 & 0 &
10\\ -4 & 7 & 8 & -10 & 0 \\ 1 & 0 & 9 & 5 & -7 \\ -2 & -9 & 0 & 6 &
-8
\end{matrix}\right),
C'=\left(\begin{matrix} 0 & -6 & 6 & -3 & 8\\  3 &  7 & 6 & 0 & 11\\
-8 & 7 & 3 & -11 & 0 \\ 6 & 0 & -2 & -7 & -7 \\ -6 & 2 & 0 & -6 & -3
\end{matrix}\right).
$$

\end{Example}

\section{From dominancy to surjectivity} \label{sec:sur}

For some polynomials $P\in \SL_d$ we can say more than in the
preceding section. Namely, we present here several cases where the
map $P\colon \Fg^d\to \Fg$ is surjective.

We start with the following simple observation (parallel to Remark 3
in \cite[\S 1]{Bo2}).

\begin{Proposition} \label{prop:domsur}
Let $P_1(X_1,\dots ,X_{d_1})$, $P_2(Y_1,\dots, Y_{d_2})$ be Lie
polynomials. Let $\Fg$ be a Lie algebra. Suppose that each of the
maps $P_i\colon \Fg^{d_i}\to \Fg$ is dominant. Let $d=d_1+d_2$,
$$P(X_1,\dots ,X_{d_1}, Y_1, \dots ,Y_{d_2})=P_1(X_1,\dots
,X_{d_1})+P_2(Y_1,\dots , Y_{d_2}).$$ Then the map $P\colon \Fg^d\to
\Fg$ is surjective.
\end{Proposition}

\begin{proof}
As the underlying variety of $\Fg$ is irreducible, the image of each
of the dominant morphisms $P_i$ $(i=1,2)$ contains a non-empty open
subset $U_i$. It remains to notice that $U_1+U_2=\Fg$ (see, e.g.,
\cite[Chapter~I, \S~1, 1.3]{Bo3}).
\end{proof}

Let us now prove surjectivity for some special maps, which are
linear in one variable.

\begin{Definition}
We call $$E_m(X, Y) =\underbrace{[[\dots [}_{m \text{
times}}X,Y],Y],\ldots, Y]\in \SL_2$$ an Engel polynomial of degree
$(m+1)$. We call $$\sum_{i = 1}^{m}a_i E_i(X,Y)\in \SL_2,$$ where
$a_i \in K$, a generalized Engel polynomial.
\end{Definition}

\begin{Theorem} \label{th:Engel}
Let $P(X,Y)\in \SL_2$ be a generalized Engel polynomial of degree
$(m+1)$, and let $P\colon  L({\sf R}, K)^2\rightarrow L({\sf R}, K)$
be the corresponding map of Chevalley algebras. If $\sf R$ does not
contain irreducible components of types listed in $\eqref{eq2.2}$
and $|K|>m|R|$, then the image of $P$ contains
$$(L({\sf R}, K)\setminus Z(L({\sf R}, K))\cup \{0\}.$$ Moreover, if
$P$ is an Engel polynomial, then the same is true under the
assumption $|K|>|R^+|$.
\end{Theorem}

\begin{proof}
Since $|K| > m|R|$, for any chosen $S\subset K$ of size $m$ there is
$h \in H$ such that $\beta (h) \notin S$ for all $\beta\in\sf R$
(see \ref{2.2}.{\bf e}). Further, for every $h\in H$ the map
$P_h\colon L({\sf R}, K)\rightarrow L({\sf R}, K)$, given by
$X\mapsto P(X,h)$, is a semisimple linear operator on $L({\sf R},
K)$ which is diagonalizable in the Chevalley basis. Each $h_\alpha$
is its eigenvector with zero eigenvalue. Further, there is a degree
$m$ polynomial $f \in K[t]$ such that $P(e_\beta, h) =
f(\beta(h))e_\beta$ for every $\beta \in {\sf R}$. (Explicitly, one
can take $f=\sum_{i=1}^m(-1)^ia_it^i$.) Define $S$ as the set of
roots of $f$ in $K$. Then $f(\beta(h)) \ne0$ for every $\beta\in \sf
R$, and therefore $\operatorname{Im}(P_h) = U$. Now the statement
follows from \ref{2.5}.{\bf f.}

If $P$ is an Engel polynomial of degree $(m+1)$, then one can take
$f = x^{m}$, and therefore $S = \{0\}$, that is, $h$ is a regular
element.  Once again, we can use \ref{2.5}.{\bf f.}
\end{proof}

\begin{cor} \label{cor:main}
Let $P= P(X,Y)\in \SL_2$ be a generalized Engel polynomial of degree
$(m+1)$, and let $\Fg$ be a simple classical Lie algebra
corresponding to the root system $\sf R$. If $|K| > m|R|$, then the
map $P\colon \Fg^2\rightarrow \Fg$ is surjective. Moreover, if $P$
is an Engel polynomial, the same is true under the assumption
$|K|>|R^+|$.
\end{cor}

\begin{remark}
Corollary \ref{cor:main} generalizes Theorem 7 of \cite{Th3} where
Question \ref{quest-lie}a was answered in the affirmative for the
words in three variables $P(X,Y,Z)$ of the form $[X,Y,\dots, Y,Z]$
and $\Fg =\Fsl(n)$.
\end{remark}

Our next result shows that one cannot hope to extend surjectivity to
central elements.

\begin{Proposition} \label{prop:noncentral}
Let $P_m(X,Y)\in \SL_2$ be an Engel polynomial of degree $m$, and
let $P\colon  L({\sf R}, K)^2\rightarrow L({\sf R}, K)$ be the
corresponding map of Chevalley algebras. Then for $m$ big enough the
image of $P$ contains no nonzero elements of $Z(L({\sf R}, K))$.
\end{Proposition}

\begin{proof} The scheme of the proof is as follows. We distinguish two cases. If
$X$ and $Y$ centralize the same element of a Cartan subalgebra of
$L({\sf R}, K)$, we prove that $P_m(X,Y)=0$ for $m$ big enough.
Otherwise, we use a partial order on $\sf R$ induced by height to
show that $P(X,Y)$ cannot lie in $H$. Here is a detailed argument.

We may assume $K$ algebraically closed. Then, by ``bringing to the
Jordan form'', we may assume $Y = h+y$ where $h \in H$, $y \in U^+$,
$[h,y]=0$. Further, let $X =  h^\prime + x$, $h^\prime \in H$, $x
\in U$.

For brevity, for every $n$ denote $z_n=P_n(X,Y)$.

{\it Case I}. $[h,x] = 0$.

Let us prove that $z_n=P_n(X,y)$ and $[z_n,h]=0$. We use induction
on $n$. For $n=1$ we have $z_1 = [X,h+y] = [X,h] + [X,y].$ Since
$[x,h]=[h',h]=0$, we have $z_1=[X,y]$. Further,
$[z_1,h]=[[X,y],h]=[[h',y],h]+[[x,y],h]$. Since
$[h',h]=[x,h]=[y,h]=0$, each summand equals zero by the Jacobi
identity, so $[z_1,h]=0$.

Assume $z_{n-1}=P_{n-1}(X,y)$ and $[z_{n-1},h]=0$. We have
$$z_n=[z_{n-1},Y]=[z_{n-1},h+y]=[z_{n-1},y]=[P_{n-1}(X,y),y]=P_n(X,y)$$
and $[z_n,h]=[[z_{n-1},y],h]=0$ by the Jacobi identity (because
$[z_{n-1},h]=[y,h]=0$).

Thus we have $P_n(X,Y) = [[X,y],y,\dots ,y]$ which is zero for $n$
big enough because $y$ is nilpotent.

\medskip

{\it Case II}. $[h,x] \ne 0$.

First suppose that $y=0$, i.e. $Y=h$ is semisimple. As $x\ne 0$, we
can write
$$
x = \sum_{\beta\in {\sf R}} f_\beta e_\beta,
$$
where $f_\beta \in K$. Since $[h,x]\ne 0$, there exists $\beta$ such
that $[h, e_\beta] \ne 0$. We now observe that if $f_\beta\ne 0$
then for every $m$ the term of $P_m(X,Y)$ containing $e_\beta$
enters with nonzero coefficient, so $P_m(X,Y)$ belongs to $U$ and
thus does not belong to the centre.

So assume $y\ne 0$ and write
$$y = \sum_{\beta\in {\sf R^+}} p_\beta e_\beta,$$
where $p_\beta \in K$.

Put
$$
{\sf R}_h = \{\beta \in {\sf R}\,\,\,\mid\,\,\,\beta (h) \ne
0\},\,\,\, \hat{{\sf R}}_h = \{\beta \in {\sf
R}\,\,\,\mid\,\,\,\beta (h) = 0\},
$$
$$
{\sf R}_x = \{\beta \in {\sf R}\,\,\,\mid\,\,\,f_\beta  \ne
0\},\,\,\, {\sf R}_y = \{\beta \in {\sf R}\,\,\,\mid\,\,\,p_\beta
\ne 0\}.
$$
All these sets are non-empty, and
$${\sf R}_{h, x} = {\sf R}_h\cap{\sf R}_x \ne \emptyset.$$
We have ${\sf R}_y\subseteq {\sf R}^+$, ${\sf R}_y\subseteq\hat{{\sf
R}}_h$.

Let $\prec$ be the (partial) order on ${\sf R}$ induced by height.
Recall that by definition $\alpha\prec\beta$ if and only if
$\beta-\alpha$ is a sum of positive roots. We fix some minimal
$\gamma$ in ${\sf R}_{h,x}$.

Further, write
$$
P_n(X, Y) = z_n =  \sum_{\beta\in {\sf R}} d_{n,\beta}e_\beta + h_n
$$
where $h_n \in H$, $d_{n, \beta}\in K$.

{\it Claim:}

a) $d_{n, \gamma} \ne 0$;

b) if $d_{n, \delta} \ne 0$ and $\delta \ne \gamma$, then either
$\delta \in \hat{{\sf R}}_h$ or $ \delta \nprec \gamma$.

Evidently, a) is enough to establish the assertion of the
proposition.

Let us prove the claim by induction on $n$. Let first $n=1$. We have
\begin{equation}
[h^\prime, y] = \sum_{\beta\in \hat{\sf R}_h} a_\beta e_\beta, \,\,
[x, y]  = \sum_{\beta\in {\sf R}} b_\beta e_\beta + h_1,\,\, [x,h] =
\sum_{\beta\in {\sf R}_h }c_\beta e_\beta, \label{eq:n=1}
\end{equation}
where $h_1\in H$, and we have  $d_{1,\beta} = a_{\beta} + b_\beta$
or $d_{1,\beta} = b_\beta + c_\beta$.

a) We have $a_\gamma = 0$, $c_\gamma \ne 0$ because $\gamma \in {\sf
R}_{h, x}\subseteq {\sf R}_h.$ Let us prove that $b_{\gamma}=0$.
Assume to the contrary that $b_\gamma \ne 0$. Then from the middle
equality in \eqref{eq:n=1} it follows that there are roots $\alpha
\in {\sf R}_x$ and $\beta \in {\sf R}_y$ such that $[e_\alpha,
e_\beta] = e_\gamma$ (and so $\gamma=\alpha +\beta$). Since
$[h,e_\beta] = 0$ and $[h, e_\gamma]\ne 0$, we have $[h, e_\alpha
]\ne 0$. Hence $\alpha \in {\sf R}_h$ and therefore $\alpha \in {\sf
R}_{h, x} = {\sf R}_h\cap {\sf R}_x.$ Since $\gamma = \alpha
+\beta$, we have the inequality $\alpha \prec \gamma$  because
$\beta$ is a positive root. This is a contradiction with the choice
of $\gamma$ (recall that $\gamma$ is a minimal root in ${\sf R}_{h,
x}$ with respect to the partial order $\prec$). Thus $b_\gamma = 0$
and $d_{1,\gamma} = c_\gamma \ne 0$.

b) Suppose $d_{1, \delta} \ne 0$ and $\delta \notin \hat{{\sf
R}}_h$. Then $\delta \in {\sf R}_h$ and $d_{1, \delta} = b_\delta +
c_\delta$. If $c_\delta \ne 0$, then  $\delta \in {\sf R}_x$. Hence
$\delta \in {\sf R}_{h, x}$ and $\delta\nprec \gamma$ because of the
choice of $\gamma$. If $c_\delta = 0$, then $d_{1, \delta} =
b_\delta \ne 0$. Then $e_\delta = [e_\alpha, e_\beta]$ for some $
\alpha \in {\sf R}_x$, $\beta \in {\sf R}_y$. Since $\delta \in {\sf
R}_h$ (and so $[h,e_{\delta}]\ne 0$) and $\beta \in {\sf R}_y
\subseteq {\hat{\sf R}}_h$ (and so $[h,e_{\beta}]=0$), we have $[h,
e_\alpha]\ne 0\Rightarrow \alpha \in {\sf R}_h\Rightarrow \alpha \in
{\sf R}_{h,x}$. Suppose that $\delta = \alpha + \beta \prec \gamma.$
Then $\alpha \prec \gamma$ which is again a contradiction with the
choice of $\gamma$. Hence $\delta \nprec \gamma$.

Let us now assume

a) $d_{n-1, \gamma} \ne 0$;

b) if $d_{n-1, \delta} \ne 0$ and $\delta \ne \gamma$, then either
$\delta \in \hat{{\sf R}}_h$ or $\delta\nprec \gamma$,

and prove the same assertions for $n$.

Consider
$$
z_n = [z_{n-1},h+y] = \underbrace{\sum_{\beta\in {\sf R}}
d_{n-1,\beta }[e_\beta, h]}_{I} + \underbrace{\sum_{\beta\in {\sf
R}} d_{n-1,\beta }[e_\beta, y]}_{II} +\underbrace{[h_{n-1},
y]}_{III}.
$$
The induction hypotheses imply that
$$
z_n = \underbrace{\sum_{\delta \in \hat{\sf R}_h}q_\delta
e_\delta}_{\spadesuit} + s_\gamma e_\gamma +
\underbrace{\sum_{\delta \in {\sf R}_h, \delta \nprec\gamma}s_\delta
e_\delta}_{\heartsuit} +h_n,
$$
where $h_n\in H$ and $s_\gamma \ne 0$. Indeed, sum I has only terms
of types $\heartsuit$ and the term $s_\gamma e_\gamma \ne 0$.
Further, sum II has terms of types $\spadesuit$ and $\heartsuit$ and
elements of $H$. Sum III has only terms of type $\spadesuit$ because
${\sf R}_y\subseteq\hat{{\sf R}}_h$. Thus conditions a) and b) hold
for $z_n$, and $z_n = P(X, Y) \notin Z(L({\sf R}, K))$.
\end{proof}

\begin{remark}
Suppose we are in one of the exceptional cases listed in
\eqref{eq2.2}. Let us exclude abelian and solvable cases 1), 2) of
Section \ref{2.3}. Also in case 6) in \thmref{th:Engel} we may
consider the Lie algebra $[L({\sf R}, K),L({\sf R}, K)]$ instead of
$L({\sf R}, K)$. In cases 3),4), 5) the algebra $L({\sf R}, K)$
contains an ideal $I$ (generated by short roots) such that the
quotient $\bar L=L({\sf R}, K)/I$ is not on list \eqref{eq2.2}, and
therefore the assertion of \thmref{th:Engel} on surjectivity of $P$
holds for $\bar L$.
\end{remark}

\begin{remark}
In the group case the phenomenon of Proposition
\ref{prop:noncentral} can be observed already for $m=1$: some
quasisimple groups contain central elements that are not
commutators, see \cite{Th1}, \cite{Dj} for infinite groups and
\cite{Bl} for finite groups.
\end{remark}

\begin{Example}
In the following example we show that non-Engel maps are not
necessarily surjective. Let $$P=P(X, Y) = [[[X,Y], X], [X, Y],
Y]]\colon \mathfrak{sl}(2, K)\times \mathfrak{sl}(2, K) \rightarrow
\mathfrak{sl}(2, K)$$ where $\operatorname{char}(K) \ne 2$ and $K$
is an algebraically closed field.

Let $\{e,f,h\}$ be the standard basis of $\Fsl(2)$: $[e,f]=h$,
$[h,e]=2e$, $[h,f]=-2f$. First note that
$P(X+mY,Y)=P(X,Y+mX)=P(X,Y)$. Therefore we may assume that
$X=ae+bf$, $Y=cf+dh$. A straightforward calculation then gives
$$
P(X,Y)=4a^2cs h+8a^2ds e + 8abds f
$$
where $s=4bd^2-ac^2$. This implies that in $\operatorname{Im}(P)$
there are no elements of the form $me$ or $mf$ with $m\ne 0$.

\end{Example}

\section{Concluding remarks and possible generalizations} \label{sec:gen}

\begin{remark} \label{rem:ass}
The method used in the proof of \thmref{th:main} (which goes back to
\cite{DS} and \cite{Bo2}) is applicable to the problem of dominancy
of polynomial maps on associative matrix algebras (which is
attributed to Kaplansky, see \cite{KBMR} and references therein).
More precisely, let $P(X_1, \dots ,X_d)\in K\left<X_1, \dots,
X_d\right>$ be an associative, noncommutative polynomial (i.e., an
element of the finitely generated free associative algebra), and let
$P\colon M_n(K)^d\rightarrow M_n(K)$ denote the corresponding map.
Then the same inductive argument as in the proof of \thmref{th:main}
shows that if $P(X_1, \dots ,X_d)$ is not identically zero on
$M_1(K)^d$ then the map $P$ is dominant for all $n$. In the
situation where $P(X_1, \dots ,X_d)$ is identically zero on $K^d$,
one can consider the induction base $n=2$ and prove that if the
restriction of $P$ to $M_2(K)^d$ is dominant then so is $P$. The
assumption made above holds, for instance, for any semi-homogeneous,
non-central polynomial having at least one $2\times 2$-matrix with
nonzero trace among its values \cite[Theorem 1]{KBMR}. If, under the
same assumptions on $P$, $\operatorname{Im}(P)$ lies in
$\mathfrak{sl}(n,K)$, then $\overline{\operatorname{Im}(P)} =
\mathfrak{sl}(n, K)$.
\end{remark}

\begin{remark}
It would be interesting to consider maps $P$ with some fixed $X_i =
A_i$. Then one could find an approach to the dominancy calculating
the differential map of $P$.
\end{remark}

\begin{remark} \label{rem:multiple}
It would be interesting to consider a more general set-up when we
have a polynomial map $P\colon L^d\rightarrow L^s$. In \cite{GR}
some dominancy results were obtained for the multiple commutator map
$P\colon L\times L^d\rightarrow L^d$ given by the formula $P(X, X_1,
\dots, X_d) = ([X, X_1], \dots, [X, X_d])$.
\end{remark}

\begin{remark} \label{rem:doubleword}
In a similar spirit, one can consider generalized word maps $w\colon
G^d\to G^s$ on simple groups. Apart from \cite{GR}, see also a
discussion of a particular case $w=(w_1,w_2)\colon G^2\to G^2$ in
\cite[Problem~1]{BGGT}.
\end{remark}

\begin{remark} \label{rem:rings}
One could try to extend some of results of this paper to the case
where the ground field is replaced with some sufficiently good ring.
One has to be careful in view of \cite{RR}: there are rings $R$ such
that not every element of $\Fsl (n,R)$ is a commutator.
\end{remark}

\begin{remark}
It would be interesting to understand the situation with {\it
infinite-dimensional} simple Lie algebras (as well as with
finite-dimensional algebras of Cartan type in positive
characteristics). The first question is whether every element of
such an algebra can be represented as a Lie product of two other
elements. Note that the question on the existence of a simple {\it
group} not every element of which is a commutator remained open for
a long time. First examples of such groups appeared in geometric
context \cite{BG}, where the groups under consideration were
infinitely generated; later on there were constructed finitely
generated groups with the same property \cite{Mu}. These are
counter-examples in very strong sense: the so-called commutator
width, defined as supremum of the minimal number of commutators
needed for a representation of a given element as a product of
commutators, may be arbitrarily large or even infinite
\cite[Theorems 4 and 5]{Mu}.
\end{remark}

\begin{remark}
Apparently, one cannot hope to extend dominancy and surjectivity
results to polynomial maps on algebras which are far from simple.
Indeed, in the group case the simplicity assumption is essential. As
pointed out to us by D.~Calegari, if $G$ is non-elementary
word-hyperbolic group and $w$ is a nontrivial word, then one cannot
hope to have a generic element of $G$ in the image of the word map
induced by $w$ (this can be deduced from \cite{CM}).
\end{remark}

\begin{remark} \label{rem:nonassoc}
One can ask questions similar to Questions \ref{quest-lie} and
\ref{quest-group} for other classes of algebras (beyond groups, Lie
algebras and associative algebras). The interested reader may refer
to \cite{Gordo} for the case of values of commutators and
associators on alternative and Jordan algebras.
\end{remark}

\noindent{\it Acknowledgements}. A substantial part of this work was
done during the visits of the first three coauthors to the MPIM
(Bonn) in 2010. Bandman, Kunyavski\u\i \ and Plotkin were supported
in part by the Minerva foundation through the Emmy Noether Research
Institute. Gordeev was supported in part by RFFI research grants
08-01-00756-a, 10-01-90016-Bel-a.

The support of these institutions is gratefully appreciated.

We also thank M.~Agranovsky, D.~Calegari, M.~Gorelik, A.~Joseph,
A.~Kanel-Belov, and A.~Premet for helpful discussions and
correspondence.

\providecommand{\bysame}{\leavevmode\hbox
to3em{\hrulefill}\thinspace}

 \end{document}